\documentclass{siamltex}
\usepackage{amsmath,amssymb}
\usepackage{graphicx}
\usepackage{amsmath}
\usepackage{amsfonts}
\usepackage{amssymb}
\usepackage{subfigure}
\usepackage{cite}
\usepackage[margin=2in]{geometry}
\usepackage{algorithm}
\usepackage{color}
\usepackage{comment}

\usepackage[colorlinks,linkcolor=blue,urlcolor=blue,citecolor=blue]{hyperref}

\usepackage[american]{babel}

\def\calH{\mathcal{H}}
\def\calQ{\mathcal{Q}}
\def\calT{\mathcal{T}}
\def\calV{\mathcal{V}}
\def\diam{\operatorname{diam}}
\def\FE{\text{FE}}
\def\IN{\mathbb{N}}
\def\IR{\mathbb{R}}

\makeatletter
\def\BState{\State\hskip-\ALG@thistlm}
\makeatother

\newtheorem{mydef}{Definition}
\newtheorem{mythm}{Theorem}
\newtheorem{mylem}{Lemma}

\newcommand{\vertiii}[1]{{\left\vert\kern-0.25ex\left\vert\kern-0.25ex\left\vert #1 
                \right\vert\kern-0.25ex\right\vert\kern-0.25ex\right\vert}}

\title{Residual-Based a posteriori error estimation for $hp$-adaptive
  finite element methods for the Stokes equations}
\author{%
  A.~Ghesmati\thanks{Department of Mathematics, Texas
    A\&M University, College Station, TX 77843-3368, USA.
    \newline\indent
     \ddag Department of Mathematics, Colorado State
     University, Fort Collins, CO 80523-1864, USA.
     \newline\indent
    \# Computational Engineering and Energy Sciences Group,
    Computational Sciences and Engineering Division,
    Oak Ridge National Laboratory, 1 Bethel Valley Rd, TN 37831, USA.
    \newline\indent
    Emails:
    \url{aghesmati@math.tamu.edu},
    \url{bangerth@colostate.edu}
    \url{turcksinbr@ornl.gov},
    }
  \and
  W.~Bangerth$\phantom{\,}^{\ddag}$
  \and 
  B.~Turcksin$\phantom{\,}^{\#}$\footnote{This manuscript has been authored by
  UT-Battelle, LLC under Contract No. DE-AC05-00OR22725 with the U.S. Department
  of Energy. The United States Government retains and the publisher, by
  accepting the article for publication, acknowledges that the United States
  Government retains a non-exclusive, paid-up, irrevocable, world-wide license
  to publish or reproduce the published form of this manuscript, or allow others
  to do so, for United States Government purposes. The Department of Energy will
  provide public access to these results of federally sponsored research in
  accordance with the DOE Public Access
  Plan (http://energy.gov/downloads/doe-public-access-plan). }
}

\begin{document}        
        \maketitle              
        
\begin{abstract}
        We derive a residual-based a posteriori error estimator for the conforming $hp$-Adaptive Finite Element Method ($hp$-AFEM) for 
        the steady state Stokes problem describing the slow motion of an incompressible fluid.
        This error estimator is obtained by extending the idea of a
        posteriori error estimation for the classical $h$-version of
        AFEM. We also establish the reliability and efficiency of the 
        error estimator. The proofs are based on the 
        well-known Cl{\'e}ment-type interpolation operator introduced in \cite{Melenk2005}
        in the context of the $hp$-AFEM. Numerical experiments show the performance of an adaptive hp-FEM algorithm using the proposed a posteriori error estimator.
\end{abstract}

\section{Introduction} \label{sec:Intro}

$h$-adaptive finite element methods -- in which the mesh size is
adjusted to resolve features of the solution -- have been known to be
efficient tools for solving partial differential equations since the
late 1970s \cite{Babuska_1978,Babuska_1979}. The development of
practical and efficient estimators of the local error over the past 25
years \cite{Verfurth-1996,Ainsworth1997,BR03} 
has made them a standard tool in the finite element analysis of many
equations and is now widely used in applications.

On the other hand, the $p$ or $hp$ versions of adaptive finite element
methods --
in which one adjusts either the polynomial degree of the approximation
on every cell, or both the polynomial degree and the mesh size -- has
seen much less practical attention. Originally introduced in
\cite{Babuska_1981,Babuska-Dorr-1981}, it is known both theoretically
and practically
that the $hp$-adaptive FEM can achieve exponential rates of
convergence with respect to the number of degrees of freedom
\cite{Schwab1998,Melenk-Schwab,Schotzau-Schwab,Costabel-Dauge}. However,
it is technically much more complicated to derive reliable and
efficient estimates of the error for $hp$ approximations. Furthermore,
even once estimates for the error on each cell are available, one is
faced with the decision whether increasing the polynomial degree $p$ of
the approximation or reducing the mesh size $h$ is more likely to
reduce the error, measured with regard to the computational cost of
the two possible resulting meshes (see, for example,
\cite{Eibner-Melenk,Wihler,Demkowicz-2002,Rachowicz-89,Ainsworth-98,Heuveline-2003,Buerg_Conv}). Finally,
the implementation of 
algorithms and data structures for conforming $hp$ finite element
methods is complex in practice \cite{BK07}.

Furthermore, it has proven to be significantly more difficult to
extend many results that are well-established for $h$ adaptivity to
$hp$ adaptivity for equations that are not as simple as the Laplace
equation. Consequently, published theoretical considerations of error
estimates and optimality of refinement strategies are still largely
confined to the Laplace equation.  Despite the known superiority of
$hp$ adaptivity in terms of computational efficiency, its practical
impact has therefore not been as profound as $h$-adaptive methods.

In this contribution, we address one of these difficulties by deriving
residual-based a posteriori error estimates for conforming $hp$
discretizations of the Stokes equation. This work is inspired by
previous work for the Laplace equation
\cite{Buerg_Conv,Dorfler2007,Melenk2001}. 
However, it has to address the
key difficulty of the Stokes equation that the solution is not the
unconstrained minimizer of an energy. Therefore, the Stokes operator is not 
positive definite, so that working with it is not as straightforward
as for example with elliptic operators with their implied coercivity condition.

In particular, we present the following results:
\begin{itemize}
\item We derive estimates for the error between the finite-dimensional
  $hp$ approximation and the continuous solution of the Stokes
  equation. 

\item As in similar approaches for the Laplace equation, it is not
  easily possible to show that these estimates are reliable and efficient,
  i.e., that the true error is bounded from above and below by our
  estimator up to a constant that does not depend on $h$ or
  $p$. This is so because the inverse estimates that are used to
  derive reliability and efficiency statements typically involve the
  polynomial degree $p$. To overcome this deficiency, we instead
  introduce a whole \textit{family} of estimates $\eta_{\alpha}$ parameterized by an
  index $\alpha\in[0,1]$. For a fixed $\alpha$, we can not show that
  an estimator is both efficient and reliable; on the other hand, we
  can show that for some members of this family, either one or the
  other property hold. However, we demonstrate through numerical experiments
  that our estimator for a given $\alpha$ is, in practice, indeed both reliable and efficient.

\item We devise a strategy to mark cells for either $h$ or $p$
  refinement based on criteria for a systematic reduction of the
  error.

\item Although we make no claims about the optimality of this
strategy -- i.e., we can not prove that among all strategies it
leads to the greatest error reduction -- we show numerical
results that suggest that the strategy can achieve the desired
exponential convergence rate for the $hp$-adaptive refinement.
\end{itemize}
To the best of our knowledge, none of these properties have previously
been derived or demonstrated for the Stokes equation using continuous $hp$-adaptive
finite element methods. (However, some related work for
\textit{dis}continuous Galerkin discretizations of the Stokes
equations is available in \cite{Houston2004}.)

The outline of the remainder of this paper is as follows:
In Section \ref{sec:model problem}, we introduce the
Stokes problem, its weak formulation and the conforming discretization
with which we intend to solve it computationally. In Section \ref{sec:
Aux_Results}, we introduce necessary notation and
state our assumptions as well as some important theoretical results (such as the
Cl\'{e}ment interpolation operator and polynomial inverse estimates)
on which we rely throughout this work. The main results are derived
in Section \ref{sec: Error Est}, where we develop an $hp$ residual-based a posteriori
error estimator for the Stokes problem, followed by the analysis of the
reliability and the efficiency of our error estimator. In Section \ref{sec:hp-Adaptive Ref}
we discuss the details of our $hp$ algorithm, i.e., the criterion upon
which we choose either $h$ or $p$ refinement. Finally in Section \ref{sec: Numerical Results} we present
numerical results and demonstrate the performance of the proposed
error estimators using practical examples.

\section{The Stokes problem and basic assumptions}\label{sec:model problem}
Let $\Omega\in\IR^{\text{2}}$ 
be an open and connected domain with smooth boundary $\Gamma=\partial\Omega$ such that it satisfies a Lipschitz condition. 
$u(\text{x})$ is the velocity 
and $\varrho(\text{x})$ be the pressure of the fluid at some point $\text{x} \in \Omega$, respectively. 

Given body forces $f \in L^2(\Omega)^{\text{2}}$ and the constant viscosity parameter $\nu>0$, consider stationary incompressible fluid flow as our model problem: For the Stokes equations, we are interested in finding $u:\Omega\to\IR^{\text{2}}$ and $\varrho:\Omega\to\IR$ such that
\begin{equation}\label{stocks}
\begin{split}
-\nu \Delta u+ \nabla \varrho &= f \qquad \text{in } \enspace \Omega,\\
-\nabla\cdot u &= 0 \qquad\text{in } \enspace \Omega,\\
u &= 0 \qquad\text{on } \enspace \Gamma.
\end{split}
\end{equation}
For ease of presentation, we here assume homogenous \textit{no slip}
boundary condition on the velocity field.
(However, similar results as the ones shown herein are also valid for other type of boundary conditions.)
To ensure uniqueness of solution, we require \textit{vanishing mean}
for pressure field, i.e., that $\int_{\Omega} \varrho = 0$. Here and
below, we limit ourselves to the two-dimensional case primarily
because Lemmas~\ref{inverse_ineq} and \ref{edge_inverseq} below are only
available for this case; however, we expect that with additional work,
all main results herein could also be shown to hold in three space dimensions.

We denote the standard Sobolev spaces by $H^m(\Omega)$ for $m\in\IN_0$. In particular, the norm and the scalar product of $L^2(\Omega)= H^0(\Omega)$ are
denoted by $\Vert\cdot\Vert_\Omega$ and $(\cdot,\cdot)_\Omega$, respectively. To account for homogeneous Dirichlet boundary conditions, we set
\begin{equation*}
H_0^1(\Omega):=\lbrace v\in H^1(\Omega): \varphi=0\text{ on }\Gamma \rbrace.
\end{equation*}
Further, we denote the space containing all functions in $L^2(\Omega)$ with zero mean value by
\begin{equation*}
L_0^2(\Omega):=\lbrace v\in L^2(\Omega): (\varphi, 1)_\Omega=0 \rbrace
\end{equation*}
and define
\begin{equation*}
\calH (\Omega):= H_0^1(\Omega)^{\text{2}} \times L_0^2(\Omega).
\end{equation*}
Then, we introduce the bilinear form $\mathcal{L}:\calH (\Omega)\times\calH (\Omega)\to\IR$ by
\begin{equation}\label{bilin}
\mathcal{L}([u,\varrho];[v,q]):= (\nu \nabla u, \nabla v)_\Omega-(\varrho, \nabla\cdot v)_\Omega- (\nabla\cdot u, q)_\Omega.
\end{equation}
The weak formulation of problem (\ref{stocks}) then seeks $[u, \varrho]\in \calH$ so that
\begin{equation}\label{weak}
\mathcal{L}([u,\varrho];[v,q])=(f,v)_\Omega \qquad \forall [v,q]\in \calH (\Omega).
\end{equation}
Due to the continuous $\inf$-$\sup$ condition 
\begin{equation*}
\inf_{[u,\varrho]\in \calH } \sup_{[v,q]\in \calH}\frac{\mathcal{L}([u,\varrho];[v,q])}{\left(\Vert\nabla u \Vert_\Omega+ \Vert \varrho \Vert_\Omega\right)\left(\Vert\nabla v \Vert_\Omega+ \Vert q \Vert_\Omega\right)}\geq\kappa>0,
\end{equation*}
where $\kappa$ is the $\inf$-$\sup$ constant depending only on $\Omega$, the weak problem is
well-posed and has a unique solution, see \cite{Brezzi1974} and
\cite{Girault1986}.

Now, assume $\calT= \lbrace K \rbrace$ is a triangulation of domain $\Omega$. For each element
$K$, we associate an element map $T_K:\hat{K}\to K$ where the
reference cell is $\hat{K}=[0,1]^{\text{2}}$. Further, we define the mesh 
size vector $h:=\left(h_K\right)_{K\in\calT},$ where $h_K:=\diam(K)$. With each element $K \in \calT$, we associate a polynomial degree $p_K \in \IN$ and collect them in a polynomial degree vector $p:=\left(p_K\right)_{K\in\calT}$. Throughout
this work, we assume that the discretization $(\calT,p)$ of $\Omega$ is $\left(\gamma_h,\gamma_p\right)$-regular \cite{Schwab1998,Szab1991}.

\begin{mydef}[$\left(\gamma_h,\gamma_p\right)$-Regularity]\label{regular_def}
A discretization $(\calT,p)$ is called $\left(\gamma_h,\gamma_p\right)$-regular if and only if there exist constants $\gamma_h,\gamma_p > 0$ such that for all $K, K' \in \calT$ with 
$\overline K\cap \overline{K'}\neq \emptyset$ there holds
\begin{equation} \label{hp-regularity}
\gamma_h^{-1}h_K\leq h_{K'} \leq \gamma_hh_K,
\qquad \text{and} \qquad
\gamma_p^{-1}p_K \leq p_K'\leq \gamma_pp_K.
\end{equation}
\end{mydef}

In other words, the condition implies that the element sizes and also
the polynomial degrees of neighboring elements are comparable.

To define the discrete solution space, for an element $K\in\calT$
denote $\mathcal{F}(K)$ the set of all interior faces of cell
$K$. Then, define by $h_{f}:=\diam(f)$ the
diameter of face $f\in\cal F(K)$ and by $p_{f}:= \max\left\{p_K,
p_{K'}\right\}$ its polynomial degree where for $K,K'\in\calT$ are
the cells adjacent to $f$.
Further, the problem is discretized by the standard $(p_{k}, p_{k-1})$
Taylor-Hood finite element. The corresponding $hp$ spaces for velocity and pressure are then
\begin{align} \label{velocity-FE-space}
V^p_u(\calT)^2 &:= \left\{u \in H_0^1(\Omega)^2:\enspace u|_{K}  \circ  T_K \in \calQ_{p_K}^2\left(\hat{K}\right) \text{ for all }K\in\calT\right\},
\\
\label{pressure-FE-space}
V^p_\varrho(\calT)&:= \left\{\varrho \in L_0^2(\Omega):\enspace \varrho|_{K}\circ T_K \in \calQ_{p_K-1}\left(\hat{K}\right)\text{ for all }K\in\calT\right\}
\\
\label{FE-space}
\calV^p(\calT) &:=V_u^p(\calT)^2\times V_\varrho^p(\calT) \subseteq \mathcal{H}(\Omega).
\end{align}
Here, $\calQ_r$ is the tensor-product polynomial space of complete degree at most $r\in\IN_0$.
Then, the discrete approximation to ($\ref{weak}$) consists of seeking $\left[u_{\FE},\varrho_{\FE}\right] \in \calV^p(\calT)$ such that
\begin{equation}\label{discrete_weak}
\mathcal{L}\left(\left[u_{\FE},\varrho_{\FE}\right];\left[v_{\FE},q_{\FE}\right]\right)=\left(f,v_{\FE}\right)_\Omega\qquad\forall\left[v_{\FE},q_{\FE}\right]\in\calV^p(\calT).
\end{equation}
This choice of spaces satisfies the discrete Babuska-Brezzi condition \cite{Arnold84}
\begin{equation*}
\inf_{[u_h,\varrho_h]\in \cal{H} } \sup_{[v_h,q_h]\in \cal{H}}\frac{\mathcal{L}([u_h,\varrho_h];[v_h,q_h])}{\left(\Vert\nabla u_h \Vert+ \Vert \varrho_h \Vert \right)\left(\Vert\nabla v_h \Vert+ \Vert q_h \Vert\right)}\geq\kappa_d>0,
\end{equation*}
where the constant $\kappa_d$ is independent of cell size $h$ and
polynomial degree $p$.
Consequently, problem \eqref{discrete_weak} is well posed.

Furthermore, Galerkin orthogonality holds:
Let $[u,\varrho]\in\calH$ be the solution of ($\ref{weak}$) and
$\left[u_{\FE},\varrho_{\FE}\right]\in\calV^p(\calT)$ be the solution
of ($\ref{discrete_weak}$), then
	\begin{equation} \label{Galerkin-Orthogonality}
	\mathcal{L}\left(\left[u-u_{\FE},\varrho-\varrho_{\FE}\right];[v_{\FE},q_{\FE}]\right)=0\qquad\forall\left[v_{\FE},q_{\FE}\right]\in\calV^p(\calT).
	\end{equation}

\section{Auxiliary results} \label{sec: Aux_Results}
We provide some auxiliary results which we use later in our work. This includes an $H^1$-conforming interpolation operator that preserves homogeneous Dirichlet boundary conditions, and some polynomial smoothing estimates. The $H^1$-conforming interpolation operator is a Cl\'{e}ment-type interpolation which replaces point evaluation
by a local average \cite{Clement1975}. The procedure does not require the extra regularity of the
point evaluation, and is consequently well-defined for functions in $H^1(\Omega)$. In \cite{Scott1990}, this interpolation operator was
modified in such a way that it also preserves polynomial boundary conditions. Melenk in \cite{Melenk2005} extended the aforementioned $H^1$-conforming interpolation
to the context of $hp$-adaptive finite element spaces.\\
In our definition of $hp$-Cl\'{e}ment interpolation operators,
consider $\mathcal{T}$ as a $(\gamma_{h}, \gamma_{p})$-regular
triangulation of $\IR^{\text{d}}$. (For cases where we would
want to impose Dirichlet boundary conditions on only a subset
$\Gamma_D\subset\Gamma$, we can require that $\Gamma_D$ can be
exactly represented by a collection of faces, i.e.,
$\bar\Gamma_D=\cup_{K \in \cal{T}} \partial K \cap\bar{\Gamma}_D$.)
Then, for a cell $K\in\calT$ and a face $f \in \mathcal{F}(K)$ we
define the patch sets
\begin{align} \label{patch}
\omega_K&:=K \cup \bigcup\{ L \in\calT: \text{$L$ shares a common edge with $K$}\},
\\
\label{patch-face}
\omega_f&:=\bigcup\{ L \in\calT : \text{$f$ is an edge of $L$}\}.
\end{align}

The following result from \cite{Melenk2001} then provides an estimate for the interpolation error in terms of the gradient of the interpolated function:
\begin{mythm}[$H^1$-Conforming Interpolation]\label{h1interpol_thm}
Let $\mathcal{T}$ be $\left(\gamma_h,\gamma_p\right)$-regular and $K \in \mathcal{T}$ be arbitrary. 
Then, there exists a bounded linear operator
$\Pi^{hp}:H_0^1(\Omega)^2\to \mathcal{V}^{p} (\mathcal{T})$ -- namely,
the Cl\'ement interpolation operator --, and
a constant $C >0$ independent of mesh size $h$ and polynomial degree
$p$ such that for all $u\in H_0^1(\Omega)$ and all
$f\in\mathcal{F}(K)$
\begin{align}\label{clem_1}
\left\Vert u-\Pi^{hp} u\right\Vert_{L^2(K)}&\leq C\frac{h_K}{p_K}\Vert \nabla
u\Vert_{L^2({\omega_K})},
\\
\label{clem_2}
\left\Vert u-\Pi^{hp} u\right\Vert_{L^2(f)}&\leq C\sqrt{\frac{h_f}{p_f}}\Vert\nabla u\Vert_{L^2({\omega_f})}.
\end{align}
\end{mythm}
\begin{proof}
Following the lines of \cite{Schwab1998}, one can find proofs in \cite[Theorem 3.3]{Melenk2005}.
\end{proof}

Next, let us present some polynomial smoothing estimates that
are widely used in the error analysis
of many numerical methods for partial differential equations and
integral equations \cite{Bernardi2001, Melenk2001}. We will later use
them in proving upper and lower bounds of our error estimator. Specifically,
define the smoothing weight functions $\Phi_{K}: K\subset \IR^{2} \to\IR^{+}$ and $\Phi_{\omega_f}:  \omega_f\subset \IR^{2}\to\IR^{+}$ by
\begin{align}\label{weightfunc}
\Phi_{K}(x) &:= \frac{1}{h_K}\operatorname{dist}\left(x,\partial K\right),
\\
\label{weightfunc_edge}
\Phi_{\omega_{f}}(x) &:= \frac{1}{\text{diam}(\omega_f)} \text{dist} (x, \partial \omega_f).
\end{align}
Then we have:
\begin{lemma}\label{inverse_ineq}
	Let $\delta \in [0,1]$, $a,b\in\IR$ such that $-1\le a\le
        b$. Then, for any $\pi_p \in \calQ_p\left( K \right)$, there
        exists some constant $C>0$ independent of $h$ and $p$ so that
	\begin{align}\label{inv-equ-01}
	\Vert \pi_p\left(\Phi_{K}\right)^a \Vert_{\text{L}^2(K)} &\leq C (a,b) p^{(b-a)}
	\Vert \pi_p\left(\Phi_{K}\right)^b \Vert_{\text{L}^2(K)},
        \\
        \label{inv-equ-02}
	\Vert \nabla \pi_p \left(\Phi_{K}\right)^{\delta} \Vert_{\text{L}^2(K)}
	&\leq \frac {C (\delta) p^{(2-\delta)}}{h_K} \Vert \pi_p \left(\Phi_{K}\right)^{\frac{\delta}{2}} \Vert_{\text{L}^2(K)}.
	\end{align}
\end{lemma}
\begin{proof}
  See \cite[Lemmas 4, 5]{Bernardi2001} and \cite [Lemma 2.5]{Melenk2001}.           
\end{proof}

The next lemma provides results for the extension of a polynomial from an edge to a domain. These estimates are used in the efficiency analysis of our error estimator.
\begin{lemma}\label{edge_inverseq}
	Let $\hat{f}$ be the edge of unit square $\hat{K}$, and $0 \le \alpha \leq 1$. 
	 $\Phi_{\omega_{\hat{f}}}$ defined as in (\ref{weightfunc_edge}) the edge $\hat{f}$
	 corresponding to the unit cell $\hat{K}$. Then there exists $C_{\alpha}>0$, such that for any polynomial $\pi_{p}
	 \in \calQ_p$ and every $\delta \in (0, 1]$, there exists some extension $v_{\hat{f}} \in H_{0}^1\left(
	 \hat{K} \right)$ so that:
	\begin{align*}
	  v_{\hat{f}}\vert_{\hat{f}}&=\pi_p\Phi_{\omega_{\hat{f}}}^{\alpha}, \hspace{3pt} v_{\hat{f}}\vert_{\partial \hat{K} \backslash \hat{f}}= 0, 
          \\
	  \Vert v_{\hat{f}} \Vert_{L^2({\hat{K}})}^2 &\leq C_{\alpha} \delta
           \Vert \pi_p \Phi^{\frac{\alpha}{2}}_{\omega_{\hat{f}}} \Vert_{L^2({\hat{f}})}^2,
          \\
	  \Vert \nabla v_{\hat{f}} \Vert_{L^2(\hat{K})}^2 &\leq C_{\alpha} (\delta  p^{2(2-\alpha)}+ \delta^{-1}) \left\Vert \pi_p \Phi^{\frac{\alpha}{2}}_{\omega_{\hat{f}}} \right\Vert^{2}_{L^2(\hat{f})}.
	\end{align*}
\end{lemma}
\begin{proof}
	See \cite[Lemma 2.6]{Melenk2001}.
\end{proof}

\section{A posteriori error estimation} \label{sec: Error Est}
A posteriori error estimates assess the error between the exact solution $[u,\varrho]\in\calH$ and its
finite element approximation
$\left[u_{\FE},\varrho_{\FE}\right]\in\calV^p(\calT)$ only in terms of
known quantities \cite{ Ern2013, Babuska1978, Verfurth1989} -- i.e., the problem
data and the approximate solution. We call a functional $\eta\left(u_{\FE},\varrho_{\FE},f\right)$ an \textit{a posteriori error estimator for the Stokes equation}, if and only if there exists a constant $C>0$ such that
  \begin{equation}\label{apost_ee}
  \left\Vert\nabla\left(u-u_{\FE}\right)\right\Vert_\Omega+\left\Vert\varrho-\varrho_{\FE}\right\Vert_\Omega\le C\eta\left(u_{\FE},\varrho_{\FE},f\right).
  \end{equation}
  Furthermore, if $\eta\left(u_{\FE},\varrho_{\FE},f\right)$ can be decomposed into localized quantities $\eta_K\left(u_{\FE},\varrho_{\FE},f\right)$, $K\in\calT$, such that
  \begin{equation}\label{global_est}
  \eta(u_{\FE},\varrho_{\FE},f)^2=\sum_{K\in\calT}\eta_{K}\left(u_{\FE},\varrho_{\FE},f\right)^2,
  \end{equation}
  then $\eta_K\left(u_{\FE},\varrho_{\FE},f\right)$ is called a \textit{local error indicator}.

Estimate (\ref{apost_ee}) is usually called a ``reliability estimate'' since it guarantees that the error is controlled by the error estimator $\eta\left(u_{\FE},\varrho_{\FE},f\right)$ up to a constant independent of mesh size $h$ and polynomial degree $p$.
Further, the local error indicators $\eta_K\left(u_{\FE},\varrho_{\FE},f\right)$ provides the basis for adaptive 
mesh refinement by identifying those cells $K\in\calT$ where the error
is large and that, consequently, should be refined locally. This
procedure is then repeated 
until
$\eta\left(u_{\FE},\varrho_{\FE},f\right)$ is smaller than a
prescribed tolerance.

Computational efficiency requires that the $\eta_K$ also satisfy some efficiency property guaranteeing that the 
upper bound (\ref{apost_ee}) is sharp and does not asymptotically
overestimate the true error. To this end, we would like to derive a
local lower bound for the energy error for every cell $K\in\calT$:
\begin{equation}
\eta_K\left(u_{\FE},\varrho_{\FE},f\right)\le C\left(\left\Vert\nabla\left(u-u_{\FE}\right)\right\Vert^2_{\omega_K}+\left\Vert\varrho-\varrho_{\FE}\right\Vert^2_{\omega_K}\right)^{1/2}.
\end{equation}

\subsection{Residual-based a posteriori error analysis} \label{sec: estimator-Err Analysis}
Let us now define a residual-based a posteriori error estimator for problem (\ref{stocks}), and derive upper and lower bounds for it
in terms of the energy error of the approximated solution. In the
spirit of \cite{Melenk2001}, we define a family of error estimators
$\eta_{\alpha}$, $\alpha \in [0,1]$.
This estimator is local, i.e.,
$\eta_\alpha^2:=\sum_{K\in\calT}\eta_{\alpha;K}^2$ and can be decomposed into cell and interface contributions:
\begin{align}
\label{residual_boun}
\eta_{\alpha;K}^2&:=\eta_{\alpha;K;R}^2+\eta_{\alpha;K;B}^2,
\\
\label{residual_term}
\eta_{\alpha;K;R}^2&:=\frac{h_K^2}{p_K^2}\left\Vert\left(I_{p_K}^Kf+\nu\Delta
u_{\FE}-\nabla\varrho_{\FE}\right)\Phi_K^{\frac\alpha2}\right\Vert^2_K+\left\Vert\left(\nabla\cdot
u_{\FE}\right)\Phi_K^{\frac\alpha2}\right\Vert^2_K,
\\
\label{bound_term-chap-3}
\eta_{\alpha;K;B}^2&:= \sum_{f\in\cal F(K)} \frac{h_f}{2 p_f}  \left\Vert \left[ \nu\frac{\partial u_{\FE}}{\partial n_K} \right] \Phi_{\omega_f}^{\frac\alpha2}\right\Vert_f^2.
\end{align}
Here, $I_{p_K}^K f$ denotes the local $L^2$-projection of $f$ into the
space of piecewise polynomials of degree $p_K$. Furthermore,
$h_f:=\text{diam}(f)$ and $p_f:=\max(p_{K}, p_{K'})$ for a face $f$
that is shared by cells $K$ and $K'$. Finally, $[\cdot]$ denotes the jump
of a quantity across a face whose outward normal relative to $K$ is
indicates by $n_K$.

In the following, we will first derive an upper bound for
the energy error in terms of the estimator $\eta_\alpha$, i.e., state a
reliability estimate. 

\begin{mythm}[Reliability]\label{relibi_error_est}
  Let $[u,\varrho]\in\calH$ and 
  $\left[u_{\FE},\varrho_{\FE}\right]\in\calV^p(\calT)$
  be the solutions of (\ref{weak}) and (\ref{discrete_weak}), respectively.
  Further, let $ \alpha\in[0,1]$ and assume that triangulation $\calT$ is $\left(\gamma_h,\gamma_p\right)$-regular. 
  Then there exists
  a constant $C_{\text{rel}}>0$ independent of mesh size vector $h$ and polynomial degree vector $p$ such that
  \begin{equation*}
    \left\Vert\nabla\left(u-u_{\FE}\right)\right\Vert_\Omega^2+\left\Vert\varrho-\varrho_{\FE}\right\Vert_\Omega^2\leq C_{\text{rel}}\sum_{K\in\calT}\left(p_K^{2\alpha}\eta_{\alpha;K}^2+\frac{h_K^2}{p_K^2}\left\Vert I_{p_K}^{K}f-f\right\Vert^2_K\right).
  \end{equation*}
  In particular, the statement provides a $p$-independent reliability bound for $\alpha=0$.
\end{mythm} 

\begin{proof}
Set $e_{\FE}:=u-u_{\FE}$ and $\epsilon_{\FE}:=\varrho-\varrho_{\FE}$. From \eqref{Galerkin-Orthogonality}, we have
\begin{align*}
  \mathcal{L}\left(\left[e_{\FE},\epsilon_{\FE}\right];[v,q]\right)
  &= \left(\nu\nabla
  e_{\FE},\nabla\left(v-\Pi^{hp}v\right)\right)_\Omega
  \\
  & \qquad
  -\left(\epsilon_{\FE},\nabla\cdot\left(v-\Pi^{hp}v\right)\right)_\Omega
  - \left(\nabla\cdot e_{\FE},q\right)_\Omega
  \\
  &= \sum_{K\in\calT}\bigg(\left(\nu\nabla
  e_{\FE},\nabla\left(v-\Pi^{hp}v\right)\right)_K  
  \\
  & \qquad\qquad
  - \left(\epsilon_{\FE},\nabla\cdot\left(v-\Pi^{hp}v\right)\right)_K 
  - \left(\nabla\cdot e_{\FE},q\right)_K\bigg),
\end{align*}
where $\Pi^{hp}:H_0^1(\Omega)^2\to \mathcal{V}^p(\mathcal{T})$ is the $H^1$-conforming interpolation operator from Theorem \ref{h1interpol_thm}. Using integration by parts
and the incompressibility condition $\nabla\cdot u=0$ yields
\begin{multline*}
    \mathcal{L}\left(\left[e_{\FE},\epsilon_{\FE}\right];[v,q]\right)
    =\sum_{K\in\calT}\Bigg(\left(f+\nu\Delta
    u_{\FE}-\nabla\varrho_{\FE},v-\Pi^{hp}v\right)_{K}
    \\
    -\left(\nabla\cdot u_{\FE},q\right)_{K}
    + \sum_{f\in\cal F(K)}\left(\left[\nu\frac{\partial u_{\FE}}{\partial n}\right],v-\Pi^{hp}v\right)_f\Bigg).
\end{multline*}
The continuous Cauchy-Schwarz inequality then results in the estimate
\begin{multline*} 
    \mathcal{L}\left(\left[e_{\FE},\epsilon_{\FE}\right];[v, q]\right)
    \le\sum_{K\in\calT}\Bigg(\left\Vert I_{p_K}^Kf+\nu\Delta
    u_{\FE}-\nabla\varrho_{\FE}\right\Vert_K\left\Vert
    v-\Pi^{hp}v\right\Vert_K 
    \\
    \qquad\qquad\qquad
    +\left\Vert\nabla\cdot u_{\FE}\right\Vert_K\Vert q\Vert_K
    +\left \Vert f-I_{p_K}^Kf\right\Vert_K\left\Vert
    v-\Pi^{hp}v\right\Vert_K
    \\
    +\sum_{f\in\cal F(K)}\left\Vert\left[\nu\frac{\partial u_{\FE}}{\partial n_K}\right]\right\Vert_f\left\Vert v-\Pi^{hp}v\right\Vert_f\Bigg).
\end{multline*}
Theorem \ref{h1interpol_thm} allows us to locally bound the
differences $v-\Pi^{hp}v$. This yields
\begin{equation*}
  \begin{split}
    \mathcal{L}\left(\left[e_{\FE},\epsilon_{\FE}\right];[v, q]\right)& \le C\sum_{K\in\calT}\Bigg(\frac{h_K}{p_K}\left\Vert I_{p_K}^Kf+\nu\Delta u_{\FE}-\nabla\varrho_{\FE}\right\Vert_K \\&+\left\Vert\nabla\cdot u_{\FE}\right\Vert_K  +\frac{h_K}{p_K}\left\Vert f-I_{p_K}^Kf\right\Vert_K
    \quad \\& + \sum_{f\in\cal F(K)}\sqrt{\frac{h_f}{p_f}}\left\Vert\left[\nu\frac{\partial u_{\FE}}{\partial n_K}\right]\right\Vert_f\Bigg)\left(\Vert\nabla v\Vert_{\omega_K}+\Vert q\Vert_K\right),
  \end{split}
\end{equation*}
which we can further estimate as follows:
\begin{multline*}
    \mathcal{L}\left(\left[e_{\FE},\epsilon_{\FE}\right];[v, q]\right)
    \\ \le
    C\left(\sum_{K\in\calT}\left(\eta_{0;K}^2+\frac{h_K^2}{p_K^2}\left\Vert
    f-I_{p_K}^Kf\right\Vert_K^2\right)\right)^{\frac12}\left(\Vert\nabla
    v\Vert_\Omega^2+\Vert q\Vert_\Omega^2\right)^{\frac12}
\end{multline*}
for some constant $C>0$ independent of mesh size vector $h$ and polynomial degree vector $p$. Moreover, for $(e_{\FE}, \varepsilon_{\FE}) \in \mathcal{H} $
we have
\begin{equation*}
  \left(\left\Vert\nabla
  e_{\FE}\right\Vert_\Omega^2+\left\Vert\epsilon_{\FE}\right\Vert_\Omega^2\right)^{\frac12}
  \leq C \sup_{[v,q]\in\calH}\frac{\mathcal{L}\left(\left[e_{\FE},\epsilon_{\FE}\right];[v,q]\right)}{(\Vert\nabla v\Vert_\Omega^2+\Vert q\Vert_\Omega^2)^{\frac12}},
\end{equation*}
  for some constant $C>0$. This implies the claimed result for $\alpha=0$. Using the inverse estimates given in Lemma \ref{inverse_ineq}, we can 
bound $\eta_{0; K}$ in terms of $\eta_{\alpha; K}$ for $\alpha\in(0,1]$ from above. Therefore, setting $a:=0$ and $b:=\alpha$ in Lemma \ref{inverse_ineq} and we get
\begin{equation*}
  \left(\left\Vert\nabla e_{\FE}\right\Vert_\Omega^2+\left\Vert\epsilon_{\FE}\right\Vert_\Omega^2\right)^{\frac12} \le C_{rel}\left(\sum_{K\in\calT}\left(p_K^{2\alpha}\eta_{\alpha;K}^2+\frac{h_K^2}{p_K^2}\left\Vert f-I_{p_K}^Kf\right\Vert_K^2\right)\right)^{\frac12}
\end{equation*}
which concludes the proof.
\end{proof}

Next, we derive an upper bound for the a posteriori error estimator $\eta_{\alpha; K}$ in terms of 
the energy error $\left\Vert\nabla\left(u-u_{\FE}\right)\right\Vert_{\omega_K}^2+\left\Vert\varrho-\varrho_{\FE}\right\Vert_{\omega_K}^2$ defined on 
the patch $\omega_K$ around cell $K$. Under mild assumptions on
the mesh, this then constitutes an efficiency estimate for the error
estimator. We will first consider the residual and jump terms
$\eta_{\alpha;K;R}, \eta_{\alpha;K;B}$ 
separately and combine the derived efficiency estimates later to
obtain an upper bound for the residual-based a posteriori error
estimator from definition (\ref{residual_boun}).

\begin{mylem}\label{lemma1}
  Let $[u,\varrho]\in\calH$,
  $\left[u_{\FE},\varrho_{\FE}\right]\in\calV^p(\calT)$, and $\calT$
  as in Theorem~\ref{relibi_error_est}, and $\alpha\in[0,1]$ be
  arbitrary. Then, there exists a constant $C>0$ independent of the
  mesh size vector $h$ and polynomial degree vector $p$ so that
  \begin{multline*}
    \eta_{\alpha;K;R}^2\leq C \bigg(p_{K}^{2(1-\alpha)}\left(\nu^2\left\Vert\nabla\left(u-u_{\FE}\right)\right\Vert_K^2+\left\Vert\varrho-\varrho_{\FE}\right\Vert^2_K\right) \\ +\frac{h_K^{2+\frac\alpha2}}{p_K^{1+\alpha}}\left\Vert f-I^K_{p_K}f\right\Vert^2_K \bigg).
  \end{multline*}
  In particular, the statement provides a $p$-independent efficiency bound
  of the cell-residual term for $\alpha=1$.
\end{mylem}
\begin{proof}
Let us write the residual-based term as
$
\eta_{\alpha;K; R}^2= \eta_{\alpha;K;R_1}^2+ \eta_{\alpha;K;R_2}^2,
$
with
\begin{equation}
  \label{eq_01}
  \begin{split}
    \eta_{\alpha;K;R_1}^2 & := \frac{h^2_K}{p^2_K}\left\Vert\left(I^K_{p_K}f+\nu\Delta u_{\FE}-\nabla\varrho_{\FE}\right)\Phi_K^{\frac\alpha2}\right\Vert^2_K,\\
    \eta_{\alpha;K;R_2}^2 & := \left\Vert\nabla\cdot u_{\FE}\Phi_K^{\frac\alpha2}\right\Vert^2_K.
  \end{split}
\end{equation}
Using the idea in \cite{Melenk2001} to build test functions, for $0<\alpha\le1$, we define the cell residual term $R_K$ as,  $R_K:=\left(I^K_{p_K}f+\nu\Delta u_{\FE}-\nabla\varrho_{\FE}\right)\Phi_K^{\alpha}\in H_0^1(K)$ and obtain
\begin{equation}
  \label{eq_02}
  \left\Vert R_K\Phi_{K}^{-\frac\alpha2}\right\Vert^2_K=\left(f+\nu\Delta u_{\FE}-\nabla\varrho_{\FE},R_K\right)_K+\left(I^K_{p_K}f-f,R_K\right)_K.
\end{equation}
With equation (\ref{weak}) and applying integration by parts, the first term reads
\begin{multline*}
  \left(f+\nu\Delta u_{\FE}-\nabla\varrho_{\FE},R_K\right)_K 
  \notag
  \\ 
  =\left(\nu\nabla\left(u-u_{\FE}\right),\nabla
  R_K\right)_K-\left(\varrho-\varrho_{\FE},\nabla\cdot R_K\right)_K 
  -\left(\nabla\cdot u,q\right)_K.
\end{multline*}
Inserting into (\ref{eq_02}) and using that $\nabla \cdot u=0$ implies
\begin{align}
  \left\Vert R_K\Phi_{K}^{-\frac\alpha2}\right\Vert^2_K & = \left(\nu\nabla\left(u-u_{\FE}\right),\nabla R_K\right)_K-\left(\varrho-\varrho_{\FE},\nabla\cdot R_K\right)_K \notag\\&\qquad +\left(I^K_{p_K}f-f,R_K\right)_K
  \notag
  \\
  & \le
  \bigg(\nu\left\Vert\nabla\left(u-u_{\FE}\right)\right\Vert_K+\left\Vert\varrho-\varrho_{\FE}\right\Vert_K
  \bigg) \left\Vert\nabla R_K\right\Vert_K 
  \notag
  \\
  \label{equ 001}
  &\qquad + \left\Vert\left(I^K_{p_K} f-f\right)\Phi_K^{\frac{\alpha}{2}}\right\Vert_K\left\Vert R_K\Phi_K^{-\frac\alpha2}\right\Vert_K.
\end{align}
Using equations (\ref{inv-equ-01}) and (\ref{inv-equ-02}) in Lemma \ref{inverse_ineq}, we can estimate
\begin{align*}
    \left\Vert\nabla R_K\right\Vert_K^2 &= \left\Vert\nabla\bigg(\left(I^K_{p_K}f+\nu\Delta u_{\FE}-\nabla\varrho_{\FE}\right)\Phi_K^{\alpha}\bigg)\right\Vert_K^2\\
    &\le 2 \left\Vert\nabla\left(I^K_{p_K}f+\nu\Delta u_{\FE}-\nabla\varrho_{\FE}\right)\Phi_{K}^{\alpha}\right\Vert^2_K \\&\qquad + 2 \left\Vert\left(I^K_{p_K}f+\nu\Delta u_{\FE}-\nabla\varrho_{\FE}\right)\Phi_K^{\alpha-1}\nabla\Phi_K\right\Vert^2_K
    \\& \le C\bigg( \frac{p_{K}^{2(2-\alpha)}}{h_K^2}\left\Vert R_K\Phi_K^{-\frac\alpha2}\right\Vert^2_K \\&\qquad + \frac{C}{h^2_K} \left\Vert\left(I^K_{p_K}f+\nu\Delta u_{\FE}-\nabla\varrho_{\FE}\right)^2\Phi_K^{2(\alpha-1)}\right\Vert_K\bigg),
\end{align*}
with some $C>0$ independent of $h$ and $p$. For the second of these
two terms, we have to distinguish between two cases. Assuming $\alpha > \frac{1}{2}$, we set $a:=2(\alpha-1)$ and $b:=\alpha$ in Lemma \ref{inverse_ineq} to get
\begin{equation*}
  \left\Vert\left(I^K_{p_K}f+\nu\Delta u_{\FE}-\nabla\varrho_{\FE}\right)\Phi_K^{\alpha-1}\right\Vert_K\le Cp_K^{1-\frac\alpha2}\left\Vert R_K\Phi_K^{-\frac\alpha2}\right\Vert_K
\end{equation*}
and inserting into the estimate above yields
\begin{equation} \label{used_in_goal_paper}
  \left\Vert\nabla R_K\right\Vert_K\le C\frac{p_{K}^{2-\alpha}}{h_K}\left\Vert R_K\Phi_K^{-\frac\alpha2}\right\Vert_K.
\end{equation}
Inequality (\ref{equ 001}) then reads as
\begin{multline*}
  \left\Vert R_K\Phi_{K}^{-\frac\alpha2}\right\Vert_K
  \\ \le C\frac{p_K^{2-\alpha}}{h_K}\bigg(\nu\left\Vert\nabla\left(u-u_{\FE}\right)\right\Vert_K+\left\Vert\varrho-\varrho_{\FE}\right\Vert_K\bigg)+h_K^{\frac\alpha2}\left\Vert I^K_{p_K} f-f\right\Vert_K,
\end{multline*}
and, after multiplying both sides by $\frac{h_K}{p_K}$ and using definition (\ref{eq_01}), we have
\begin{multline}\label{equ 004}
\eta_{\alpha;K;R_1} \leq
Cp_K^{1-\alpha}\bigg(\nu\left\Vert\nabla\left(u-u_{\FE}\right)\right\Vert_K+\left\Vert\varrho-\varrho_{\FE}\right\Vert_K\bigg)
\\+\frac{h_K^{1+\frac\alpha2}}{p_K}\left\Vert I^K_{p_K} f-f\right\Vert_K.
\end{multline}
Now, let us consider the case $0\le\alpha\le\frac12$. Let $\beta:=\frac{1+\alpha}2$. Again, using the smoothing estimates given in Lemma \ref{inverse_ineq} and considering the fact that $\beta>\alpha$, we find
\begin{equation*}
  \begin{split}
    \left\Vert R_K\Phi_K^{-\frac\alpha2}\right\Vert_K &\le C p_K^{\beta-\alpha}\left\Vert\left(I_{p}^{K}f+\nu\Delta u_{\FE}-\nabla\varrho_{\FE}\right)\Phi_K^{\frac\beta2}\right\Vert_K\\
    & = C\frac{p_K^{1+\beta-\alpha}}{h_K}\eta_{\beta;K;R_1}.
  \end{split}
\end{equation*}
Estimate (\ref{equ 004}) then implies
\begin{equation*}
\begin{split}
  \left\Vert R_K\Phi_K^{-\frac\alpha2}\right\Vert_K\le C\bigg( &\frac{p_K^{2-\alpha}}{h_K}\left(\nu\left\Vert\nabla\left(u-u_{\FE}\right)\right\Vert_K+\left\Vert\varrho-\varrho_{\FE}\right\Vert_K\right)\\& +\frac{h_K^{\frac\beta2}}{p_K^{\alpha-\beta}}\left\Vert I^K_{p_K} f-f\right\Vert_K \bigg).
  \end{split}
\end{equation*}
Then, the definition of $\beta$ yields
\begin{equation}
        \begin{split}
  \eta_{\alpha;K;R_1}\le C\bigg( & p_K^{1-\alpha}\left(\nu\left\Vert\nabla\left(u-u_{\FE}\right)\right\Vert_K+\left\Vert\varrho-\varrho_{\FE}\right\Vert_K\right)\\& +\frac{h_K^{\frac{5+\alpha}4}}{p_K^{\frac{1+\alpha}2}}\left\Vert I^K_{p_K} f-f\right\Vert_K\bigg).
   \end{split}
\end{equation}
To obtain the upper bound for $\eta_{\alpha;K;R_2}^2$, we observe
\begin{equation*}
  \begin{split}
    \eta_{\alpha;K;R_2} = \left\Vert (\nabla\cdot u_{\FE} )\Phi_K^{\frac\alpha2}\right\Vert_K
    \le h_K^{\frac\alpha2}\left\Vert\nabla\cdot u_{\FE}\right\Vert_K.
  \end{split}
\end{equation*}
Since $\nabla\cdot u=0$, we have $\nabla\cdot u_{\FE}=\nabla\cdot\left(u-u_{\FE}\right)$ and, hence,
\begin{equation}
  \label{resid_02}
  \eta_{\alpha;K;R_2}\le h_K^{\frac\alpha2}\left\Vert\nabla\left(u-u_{\FE}\right)\right\Vert_K.
\end{equation}
Finally, combining estimates (\ref{equ 004}) and (\ref{resid_02}) gives the desired result.
\end{proof}

Similarly, we can derive matching estimates for the jump-based term $\eta_{\alpha ;K;B}$ in equation
(\ref{bound_term-chap-3}):
\begin{mylem}
  \label{lemma_02}
  Let $[u,\varrho]\in\calH$,
  $\left[u_{\FE},\varrho_{\FE}\right]\in\calV^p(\calT)$, and $\calT$
  as in Theorem~\ref{relibi_error_est}. Let $\alpha\in[0,1]$.
  Then, there exists some constant $C>0$ independent of mesh size vector $h$ and polynomial degree vector $p$ such that
  \begin{equation*}
  \begin{split}
    \eta_{\alpha ;K;B}^{2} \leq C\bigg( & p_K^{\frac{3-\alpha}2}\left(\nu^2\left\Vert\nabla\left(u-u_{\FE}\right)\right\Vert^2_{\omega_K}+\left\Vert\varrho-\varrho_{\FE}\right\Vert^2_{\omega_K}\right)\\& +\frac{h^2_K}{p_K^{\frac{3+\alpha}2}}\left\Vert I^K_{p_K}f-f\right\Vert_{\omega_K}^2\bigg).
    \end{split}
  \end{equation*}
\end{mylem}
\begin{proof}
 For a given element $K \in \calT$ and an interior face $f \in \cal F (K)$, there
  exists some $K_1 \in \calT$ such that $f= \partial K \cap\partial
  K_1$ and a face patch $\omega_f$ as given in
  (\ref{patch-face}). Moreover, by Lemma \ref{inverse_ineq} there exists an
  extension function $R_f \in H_0^{1}(\omega_f)$ such that $R_{f}
  \vert_{f}=\left[ \nu\frac{\partial u_{\FE}}{\partial
  n}\right]\Phi_{\omega_f}^{\alpha}$ that is continuous on $K$,
  vanishes on $\partial \omega_{f}$, and can be extended by zero to
  all of $\Omega$. Thus, we can consider $R_{f} \in H_0^1(\Omega)$. Now, to derive an upper bound for the jump-based term $\eta_{\alpha ;K;B}^{2}$, we use integration by parts to get
\begin{equation*}
\begin{split}
\left\Vert R_f \Phi_{\omega_{f}}^{-\frac{\alpha}{2}}\right\Vert^2_{f} & =(\nu \Delta u_{\FE}, R_{f})_{\omega_f}+(\nu \nabla u_{\FE}, \nabla R_f)_{\omega_{f}}.
\end{split}
\end{equation*}
From the weak formulation (\ref{weak}) we have
 \begin{equation*}
\begin{split}
\left\Vert R_f \Phi_{\omega_{f}}^{-\frac{\alpha}{2}}\right\Vert^2_{f} & = (\nu \Delta u_{\FE}, R_f)_{\omega_f}-\left(\nu \nabla\left(u-u_{\FE}\right), \nabla R_f\right)_{\omega_f}+(f, R_f)_{\omega_f} \\ &\qquad +(\varrho, \nabla \cdot R_f)_{\omega_f} +
(\nabla \cdot u, R_f)_{\omega_f}\\
  &= (\nu \Delta u_{\FE}, R_f)_{\omega_f}-\left(\nu \nabla\left(u-u_{\FE}\right), \nabla R_f\right)_{\omega_f}+(f, R_f)_{\omega_f} \\&\qquad+ \left(\varrho_{\FE},\nabla\cdot R_f\right)_{\omega_f} +\left(\varrho-\varrho_{\FE}, \nabla \cdot R_f\right)_{\omega_f},
\end{split}
\end{equation*}
using again $\nabla \cdot u =0$. Then, performing integration by parts gives
\begin{align}
    \left\Vert R_f
    \Phi_{\omega_{f}}^{-\frac{\alpha}{2}}\right\Vert^2_{f} &=
    \left(I^K_{p_K}f + \nu \Delta u_{\FE}- \nabla \varrho_{\FE},
    R_f\right)_{\omega_f}-(\nu \nabla (u-u_{\FE}), \nabla R_e)_{K_e} 
    \notag
    \\&\qquad +(\varrho-\varrho_{\FE}, \nabla \cdot R_f)_{\omega_f} +
    \left(f- I^K_{p_K}f, R_f\right)_{\omega_f}
    \notag
    \\
    \label{eq_03}
    & \le  \left(\left\Vert I^K_{p_K}f 
   + \nu \Delta u_{\FE}-\nabla\varrho_{\FE}\right\Vert_{\omega_f}+
   \left\Vert f- I^K_{p_K}f\right\Vert_{\omega_f}\right)\left\Vert
   R_e\right\Vert_{\omega_f} 
   \notag
   \\& \qquad+
  \nu \left\Vert \nabla(u-u_{\FE})\right\Vert_{\omega_f} \left\Vert \nabla R_f \right\Vert_{\omega_f} 
  + \left\Vert \varrho-\varrho_{\FE} \right\Vert_{\omega_f} \left\Vert \nabla \cdot R_e\right\Vert_{\omega_f}.
\end{align}
We again distinguish two cases. First, if $\alpha>\frac{1}{2}$, we use
Lemma \ref{edge_inverseq} and obtain the following upper bounds for $\Vert R_f \Vert_{\omega_f}$ and $ \Vert \nabla R_f \Vert_{\omega_f}$ on face $f$: 
\begin{align*}
  \Vert \nabla R_f \Vert^2_{\omega_f} &\leq C\frac{\delta p_K^{(2(2-\alpha))} +\delta^{-1}}{h_K} \left \Vert  \left[ \nu\frac{\partial u_{\FE}}{\partial n}\right]\Phi_{\omega_f}^{\frac{\alpha}{2}} \right \Vert_{f}^{2},\\
 \Vert R_f \Vert^2_{\omega_f} &\leq C\delta h_K  \left \Vert  \left[ \nu\frac{\partial u_{\FE}}{\partial n} \right]\Phi_{\omega_{f}}^{\frac{\alpha}{2}} \right \Vert_{f}^{2}.
\end{align*}
Knowing that $\Vert \nabla \cdot R_f \Vert_{\omega_f} \leq \Vert \nabla R_f \Vert_{\omega_f}$, estimate (\ref{eq_03}) yields
\begin{align*}
  \left\Vert\left[\nu\frac{\partial u_{\FE}}{\partial n}\right]\Phi_{\omega_f}^{\frac\alpha2}\right\Vert_f  &\le C \bigg(\left(\delta h_K\right)^{\frac12}\bigg(\left\Vert I^K_{p_K}f+\nu \Delta u_{\FE}-\nabla \varrho_{\FE}\right\Vert_{\omega_f} \\&\qquad+\left\Vert f- I^K_{p_K}f\right\Vert_{\omega_f}\bigg)\\
  & \qquad + \sqrt{\frac{\delta p_K^{2(2-\alpha)}+\delta^{-1}}{h_K}}\bigg(\nu\left\Vert\nabla\left(u-u_{\FE}\right)\right\Vert_{\omega_f}\\&\qquad\qquad +\left\Vert\varrho-\varrho_{\FE}\right\Vert_{\omega_f}\bigg) \bigg),
\end{align*}
and it follows with Lemma \ref{lemma1} that
\begin{align*}
  \left\Vert\left[\nu\frac{\partial u_{\FE}}{\partial n}\right]\Phi_{\omega_f}^{\frac\alpha2}\right\Vert_f  &\le C\bigg\{\left(\delta h_K\right)^{\frac12}\bigg[\frac{p_K^2}{h_K}\bigg(\nu\left\Vert\nabla\left(u-u_{\FE}\right)\right\Vert_{\omega_f}\\\qquad&+\left\Vert\varrho-\varrho_{\FE}\right\Vert_{\omega_f}\bigg)+p_K^{\frac12}\left\Vert f-I^K_{p_K}f\right\Vert_{\omega_f}\bigg] \\\qquad&
  + \sqrt{\frac{\delta p_K^{2(2-\alpha)}+\delta^{-1}}{h_K}}\bigg(\nu\left\Vert\nabla\left(u-u_{\FE}\right)\right\Vert_{\omega_f} \\&\qquad\qquad +\left\Vert\varrho-\varrho_{\FE}\right\Vert_{\omega_f}\bigg)\bigg\}.
\end{align*}
By squaring both sides and summing over all edges $f\in\cal F(K)$, we get
\begin{equation}\label{equ_02}
\begin{split}
\eta_{\alpha;K;B}^2  \le C\delta & \bigg[p_K^3\left(\nu^2\left\Vert\nabla\left(u-u_{\FE}\right)\right\Vert_{\omega_K}^2+\left\Vert\varrho-
\varrho_{\FE}\right\Vert_{\omega_K}^2\right)+h_K^2\left\Vert f-I^K_{p_K}f\right\Vert_{\omega_K}^2\\
  &\quad+ \frac{p_K^{2(2-\alpha)}+\delta^{-2}}{p_K}\left(\nu^2\left\Vert\nabla\left(u-u_{\FE}\right)\right\Vert_{\omega_K}^2+\left\Vert\varrho-\varrho_{\FE}\right\Vert_{\omega_K}^2\right)\bigg].
\end{split}
\end{equation}
Setting $\delta:=p_K^{-2}$ gives the desired result.

For $0 \leq \alpha \leq \frac{1}{2}$, similar to 
the proof of Lemma \ref{lemma1}, we set $\beta:= \frac{1+\alpha}{2}$ and apply Lemma \ref{inverse_ineq} to 
get $\eta_{\alpha;K;B}\le p_K^{\beta-\alpha}\eta_{\beta;K;B}$. Then, using (\ref{equ_02}) gives
\begin{align*}
  \eta_{\alpha;K;B}^2  &\le C\delta   \bigg[p_K^{\frac{7-\alpha}2}\left(\nu^2\left\Vert\nabla\left(u-u_{\FE}\right)\right\Vert_{\omega_K}^2+\left\Vert\varrho-\varrho_{\FE}\right\Vert_{\omega_K}^2\right) \\&\qquad +\frac{h_K^2}{p_K^{\frac{\alpha-1}2}}\left\Vert f-I^K_{p_K}f\right\Vert_{\omega_K}^2\\
  &\qquad+\frac{p_K^{2(2-\alpha)}+\delta^{-2}}{p_K^{\frac{1+\alpha}2}}\left(\nu^2\left\Vert\nabla\left(u-u_{\FE}\right)\right\Vert_{\omega_K}^2+\left\Vert\varrho-\varrho_{\FE}\right\Vert_{\omega_K}^2\right) \bigg].
\end{align*}
Again setting $\delta:=p_K^{-2}$ concludes the proof.
\end{proof}\\
Lemmas \ref{lemma1} and \ref{lemma_02} combine to yield
the desired ``efficiency'' upper bound for the error estimator $ \eta
$ in terms of the quasi-local energy error.
\begin{mythm}[Efficiency] \label{reli_effici_thm} 
  Let $[u,\varrho]\in\calH$,
  $\left[u_{\FE},\varrho_{\FE}\right]\in\calV^p(\calT)$, and $\calT$
  as in Theorem~\ref{relibi_error_est}, and $\alpha\in[0,1]$ be
  arbitrary.
  Then, there exists some constant $C_{\text{eff}}> 0$ independent of mesh size vector $h$ and polynomial degree vector $p$ such that   
\begin{equation*}
\begin{split}
\eta_{\alpha ;K}^{2}\leq C_{\text{eff}} & \bigg(p_K^k\left(\nu^2\left\Vert\nabla\left(u-u_{\FE}\right)\right\Vert^{2}_{\omega_K}+\left\Vert\varrho-\varrho_{\FE}\right\Vert^2_{\omega_K}\right)\\& +\frac{h^2_K}{p_K^{1+\alpha}}\left\Vert I^K_{p_K}f-f\right\Vert_{\omega_K}^2\bigg)
\end{split}
\end{equation*}
for all $K\in\calT$, where $k:=\max\left\{2(1-\alpha),\frac{3-\alpha}2\right\}$.
By assuming that each cell is only part of a bounded number of cell
patches, the efficiency upper bound also holds for the entire
estimator $\eta_\alpha$.
\end{mythm}

\section{$hp$-adaptive refinement} \label{sec:hp-Adaptive Ref}

To define a fully automatic $hp$-adaptive finite element algorithm, we
base our approach on the error estimator introduced in
Section~\ref{sec: estimator-Err Analysis}. It consists of the standard adaptive loop
\begin{equation}
\text{SOLVE} \longrightarrow \text{ESTIMATE} \longrightarrow \text{MARK} \longrightarrow \text{REFINE}.
\end{equation}
Of concern in this section is only the marking strategy for the third
step (given an estimate of the error as derived previously), for which
we follow the ideas of \cite{Burg2012, Buerg_Conv}. We then apply
either the usual bisection strategy of marked cells for mesh
refinement followed by ensuring that there is only one hanging node
per edge ($h$ refinement), or increase the polynomial degree (if $p$
refinement is favored).

The question in marking is whether to perform $h$- or
$p$-refinement. In both cases, one can also ask how exactly a cell is
to be subdivided, or by how much the polynomial degree should be increased. Unfortunately, the size of the estimated error
$\eta_K$ by itself is not enough to tell us which option is
to be preferred. Rather, we should estimate the error one would
``expect'' after each of these choices, and balance this information
against the cost of each choice.

%

\subsection{Convergence indicators} \label{sec: Conv-Indicator}
Let $j \in \{1, 2, \cdots, n\}$, where $n$ indicates the number of different $h$ and $p$ refinement patterns, and let $K \in \calT_N$ be a cell during the $N$-th cycle of refinement. Following \cite{Dorfler2007}, we define 
a ``convergence indicator'' $k_{K,j}\ge 0$
that estimates the error reduction on cell $K$ (relative to the
current estimated error $\eta_K$) if $K$ were refined by refinement pattern $j$.
For the Stokes problem, similar to \cite{Ainsworth1997}, we generate
this estimate by measuring the residual in a norm equivalent to the
norm on the dual of $\cal H(\omega_K)$. Let $e:= u - u_{\FE}$ and $E:= \varrho-\varrho_{\FE}$ such that $(e,E) \in \cal H$. Considering the residual 
of the Stokes problem on the local patch domain $\omega_K$, and
notation from $(\ref{bilin})$, then we have for all $(v,q)\in \cal H$:
\begin{equation*}
\int_{\omega_K} v f -\int_{\omega_K} \nabla v \cdot \nabla u_{\FE} +\int_{\omega_K} (\nabla \cdot v)\varrho_{\FE}+ \int_{\omega_K} q \nabla \cdot u_{\FE} = \mathcal{L}([v,q];[e, E])_{\omega_K}.
\end{equation*}
Integration by parts gives
\begin{equation*}
\int_{\omega_K} v\left (f+\nu\Delta u_{\FE}-\nabla\varrho_{\FE}\right)-\int_{\omega_K} q\left( \nabla\cdot u_{\FE} \right)= \mathcal{L}([v,q];[e, E])_{\omega_K}.
\end{equation*}
The pair $(w_u ,w_\varrho)\in \cal H$ is defined to be the Ritz projection of the residual, as follows:
\begin{equation}\label{ritz-rep}
(\nabla v ,\nabla(w_u))_{\omega_K}+(q , w_\varrho)_{\omega_K} = \mathcal{L}([v,q];[e, E])_{\omega_K},\hspace{20pt} \forall (v,q)\in \cal{H}.
\end{equation}
Existence and uniqueness of $(w_u, w_{\varrho})$ follows from the
continuity of the operators in the definition of the bilinear form in
$(\ref{bilin})$. The energy norm of the error can then be defined as
\begin{equation} \label{norm-equality}
\vertiii{(e, E)}_{\omega_K}^2 = \Vert  \nabla(w_u) \Vert_{\omega_K}^2+ \Vert w_\varrho \Vert_{\omega_K}^2. 
\end{equation}
Of course, this pair of functions can not be found
analytically -- we need to approximate it by solving
a discrete problem for $(w_u^j,w_\rho^j)$ using either a finer mesh, or a
finite element space with a higher polynomial degree -- i.e., one of
the choices $j$ for refinement.
For cell $K$ refined by pattern $j$, we combine the idea of the
convergence estimator in \cite{Dorfler2007} and the above discussion
on the Ritz representation of the residual \eqref{ritz-rep} and
define
\begin{equation}\label{conv-est}
k_{K,j}= \frac{1}{\eta_{K}(u_{\FE}, \varrho_{\FE})} \left(\left\Vert \nabla w_u^{j} \right\Vert^2_{\omega_K}+
\left\Vert w_{\varrho}^{j} \right\Vert^2_{\omega_K} \right)^{\frac{1}{2}}.
\end{equation}
The convergence estimator $k_{K,j}$ as defined in $(\ref{conv-est})$ indicates which
refinement pattern $j$ provides the biggest error
reduction on every cell. In order to choose the most efficient refinement
pattern, we need to balance this reduction against
a workload number $\varpi_{K,j} > 0$ that
indicates the work required to achieve the error reduction $k_{K,j}$ on cell $K$.
This workload number can be defined in a variety of ways; here, we
take it as the number of degrees of freedom in the local finite
element space, i.e.,
$\varpi_{K,j} = \text{dim}\;\calV ^{p}_{K, j}(\calT_N\rvert_{ \omega_K})$.

For each cell $K$, we then define $j_K$ to be that refinement strategy
that maximizes the expected (normalized) relative error reduction, i.e., $j_K =
\arg\max_{j \in \{1, 2, \cdots, n \}} \frac{k_{K,j}}{\varpi_{K,j}}$.

For the purpose of this work, we only consider two refinement
patterns, $j \in \{1, 2\}$, namely isotropic $h$-refinement, 
and $p$-refinement by increasing the polynomial degree by one, but the
strategy above is clearly applicable also to more general choices.

%

\subsection{Marking}\label{sec:mark-primal-hp}
We still have to decide \textit{which} cells should be refined using
the strategies $j_K$ defined above. To this end, we
seek that set 
$\mathcal{M}\subseteq \mathcal{T}$ of minimal cardinality so that
\begin{equation} \label{constraint}
\sum_{K\in \cal M} k^2_{K, j_K} \eta^2_{K} \geq \theta^2 \eta^2.
\end{equation}
We solve this problem approximately using a greedy strategy, i.e.,
using D\"{o}rfler marking. It is known, see \cite{Dorfler2007}, that
such an $\mathcal M$ exists if $\theta$ is chosen small enough.

\section{Numerical results} \label{sec: Numerical Results}
Our numerical verification of the algorithms proposed above are
implemented using the software library deal.II \cite{dealii,dealII85}.
In particular, we will keep track of the estimated error and
demonstrate that it decreases with the same asymptotic rate as the
actual error in the energy norm on a sequence of non-uniform,
$hp$-adaptively refined meshes. The \textit{effectivity index}
$I_\text{eff}$ then measures the quality of the estimator $\eta$:
 \begin{equation} \label{eff-Index}
 I_\text{eff} := \frac{\text{error estimator}}{\text{energy error}} = \frac{\eta\left(u_{\FE},\varrho_{\FE},f\right)}{\bigg(\left\Vert\nabla\left(u-u_{\FE}\right)\right\Vert^2_\Omega+\left\Vert\varrho-\varrho_{\FE}\right\Vert^2_\Omega\bigg)^{1/2}}.
 \end{equation}
 Ideally, one would want to have $I_\text{eff} =1$ as $h \rightarrow 0$; however, the
 equivalence of $\eta$ and the error in Section~\ref{sec: Error Est} has only
 been shown up to unknown constants, and consequently in practice
 we will be content if $C_1 \leqslant  I_\text{eff}  \leqslant C_2$ for some $C_1 , C_2 > 0$.
 
 \subsection{Example 1} \label{sec: Example-1}
    Let us consider a domain
    $
    \Omega= (-1,1)^2 \;\backslash\; ([0,1]\times  [-1,0])
    \subset \IR^2
    $ shaped like an ``L'', and choose the right hand side $f$ as well as inhomogeneous Dirichlet
    boundary conditions for $u$ so that the solution of the Stokes
    equations equals the smooth functions
    \begin{equation*}
    u=
    \left[
    \begin{matrix}
    -e^{x} (y \cos(y) +\sin(y))
    \\
    e^x y \sin(y)
    \end{matrix}
    \right],
    \qquad 
    \varrho= 2e^x \sin(y) - \frac 23 (1-e) (\cos(1)-1)).
    \end{equation*}       

In the following experiment, we start with a triangulation
$\mathcal{T}_0$ consisting of 12 uniform cells, and initially choose
$\calQ_3^2\times \calQ_2$ elements on all cells. We then start the
adaptive mesh iteration as discussed previously with $\theta=0.75$.

   \begin{figure}
        \centering
        \begin{tabular}{cc}
                \includegraphics[width=0.45\linewidth]{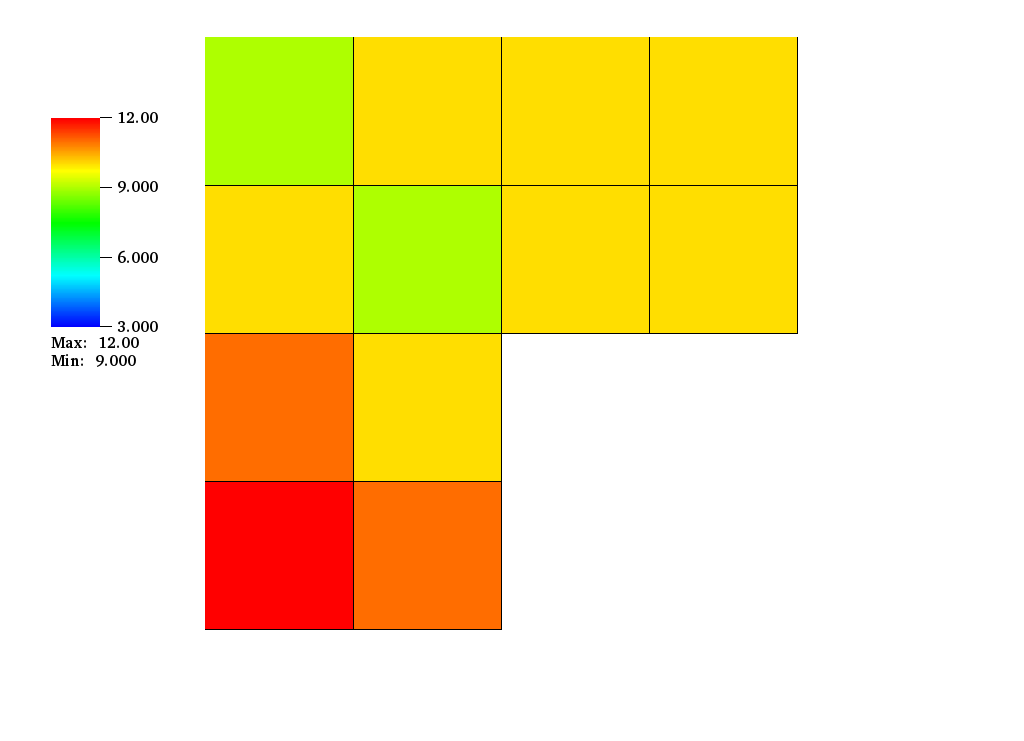} 
                &
                \includegraphics[width=0.45\linewidth]{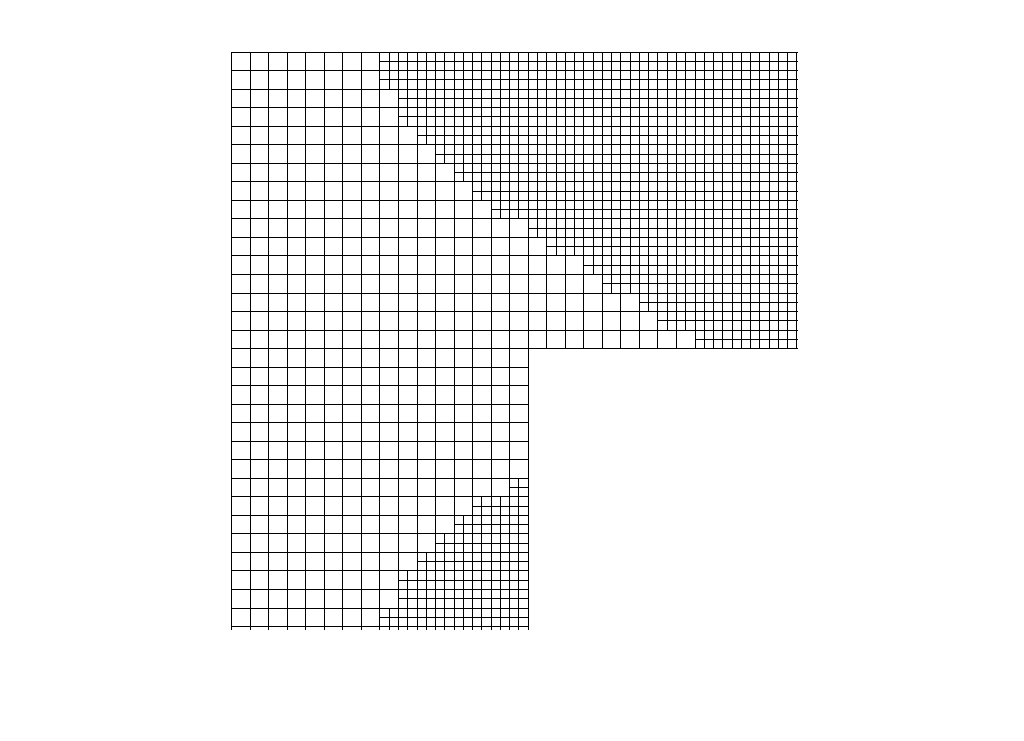} 
        \end{tabular}
        \caption{Example 1. Left: Final mesh after 11 $hp$-adaptive
        refinement steps, with color indicating the polynomial
        degrees. Right: Final mesh after 7 $h$-adaptive refinement
        steps. Both meshes have approximately the same number of
        degrees of freedom.}\label{h_and-hp-ref}
   \end{figure}


Fig.~\ref{h_and-hp-ref} shows meshes after a number of cycles if
$hp$-refinement is allowed, or if we only do
$h$-refinement. Unsurprisingly, and confirming expectations, given the smooth nature of the
exact solution, the $hp$-adaptive strategy consistently chooses $p$-refinement.
Fig.~\ref{error-errorest-eff-index} presents the
decay of the energy error and the a posteriori error
estimator as a function of number of degrees of freedom. The graph
both demonstrates the
exponential convergence rate, and also that the $hp$-error estimator is a sharp
upper bound for the energy error -- validating this as an efficient and
reliable a posteriori error estimator. We observe from the effectivity index
graph in Fig.~\ref{error-errorest-eff-index} that the $I_\text{eff}$
remains bounded in the range $5.4 \leqslant I_\text{eff} \leqslant
8.1$.

The figure's right panel also shows a comparison of errors for
$h$- and $hp$-adaptive refinement strategies. This
plot clearly shows the superiority of $hp$-AFEM over the $h$-AFEM.

 
 \begin{figure}
        \centering
        \includegraphics[width=0.32\linewidth]{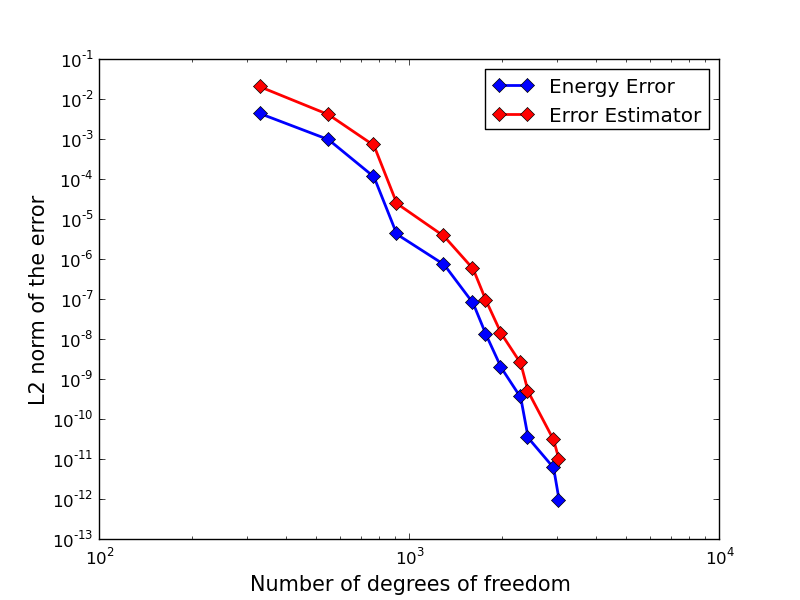} 
        \hfill
        \includegraphics[width=0.32\linewidth]{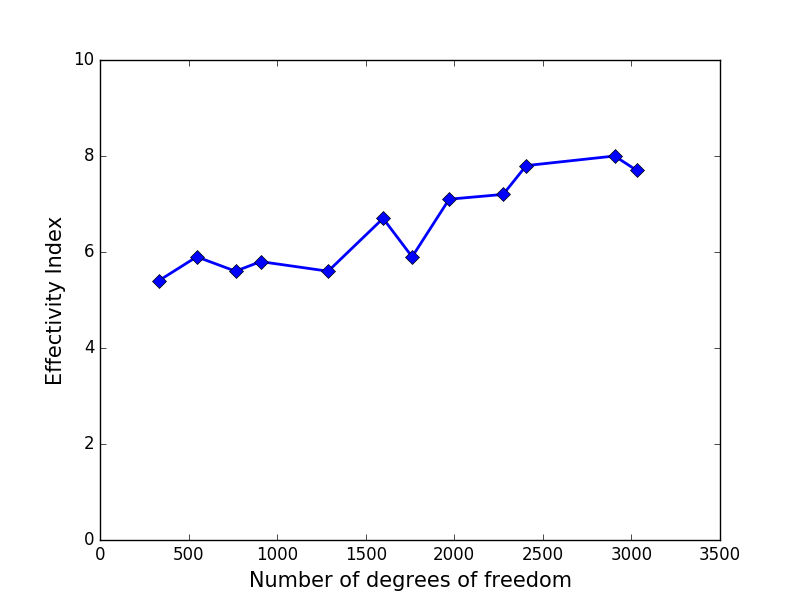}
        \hfill
        \includegraphics[width=0.32\linewidth]{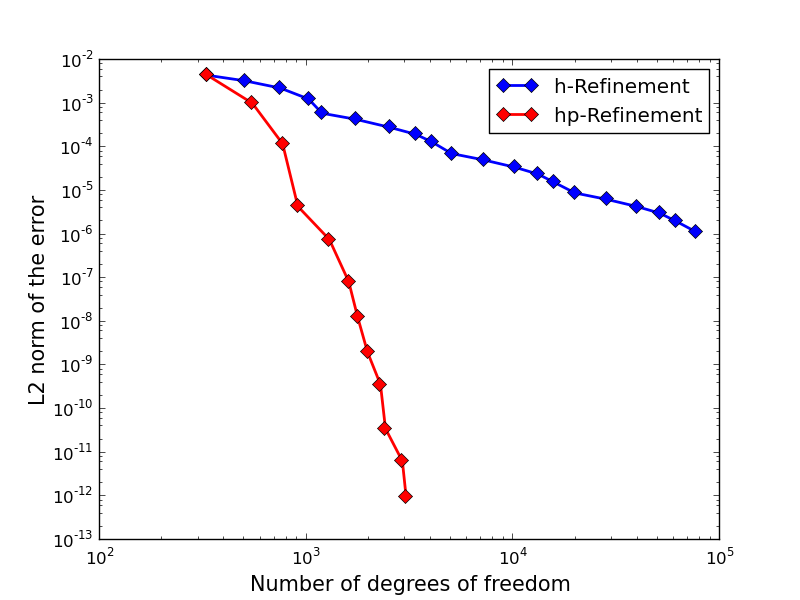}
        \caption{Example 1. Left: Comparison of the energy error and
          the error estimator for an $hp$-adaptive
          computation. Center: Effectivity indices $I_\text{eff}$ for
          the same computation. Right: A comparison of the errors for
          this computation with the errors obtained through pure
          $h$-refinement.}
        \label{error-errorest-eff-index}
 \end{figure}

\subsection{Example 2} \label{sec: Example-2}
On the same $L$-shaped domain, we now choose right hand side and
boundary values so that we reproduce the singular solution of
\cite{Dauge1989}, which reads in polar coordinates as
\begin{align*}
u(r, \varphi) &= r^{\alpha}
\left[
\begin{matrix}
\cos(\varphi)\psi^{'}(\varphi)+(1+\alpha)\sin(\varphi)\psi(\varphi)\\
\sin(\varphi)\psi^{'}(\varphi)-(1-\alpha)\cos(\varphi) \psi(\varphi)
\end{matrix}
\right ],
\\
\varrho(r, \varphi) &= -r^{\alpha-1} \frac{(1+\alpha)^2 \psi^{'}(\varphi)+ \psi^{'''}(\phi) }{1-\alpha},
\end{align*}
where
\begin{multline*}
  \psi(\varphi) =  \frac{\sin((1+\alpha)\varphi) \cos(\alpha \omega)}{1+\alpha}
  -\cos((1+\alpha)\varphi)
  \\
  -\frac{\sin((1-\alpha)\varphi) \cos(\alpha\omega)}{1-\alpha} +
  \cos((1-\alpha)\varphi), 
\end{multline*}
and $\omega= \frac{3 \pi}{2}$.
Here $\alpha$ is the smallest positive solution of
$
\sin(\alpha\omega) + \alpha \sin(\omega)=0
$
and is $\alpha\approx 0.54448373678246$.
We choose the same initial triangulation, but this time start with
$\calQ_2^2\times \calQ_1$ elements on all cells. We use $\theta=0.85$.  

Fig.~\ref{Ex-2-h_and-hp-ref} again shows $hp$- and $h$-adaptively refined
meshes generated by our error estimator. The corner singularity in the
solution is apparent.
A comparison of error and error estimator, efficiency indices, and a
comparison between $hp$- and $h$-adaptively refinement strategies is
shown in Fig.~\ref{error-errorest-eff-index-Ex2}. In particular, the
efficiency indices again remain bounded.


 \begin{figure}
        \centering
        \begin{tabular}{cc}
                \includegraphics[width=0.45\linewidth]{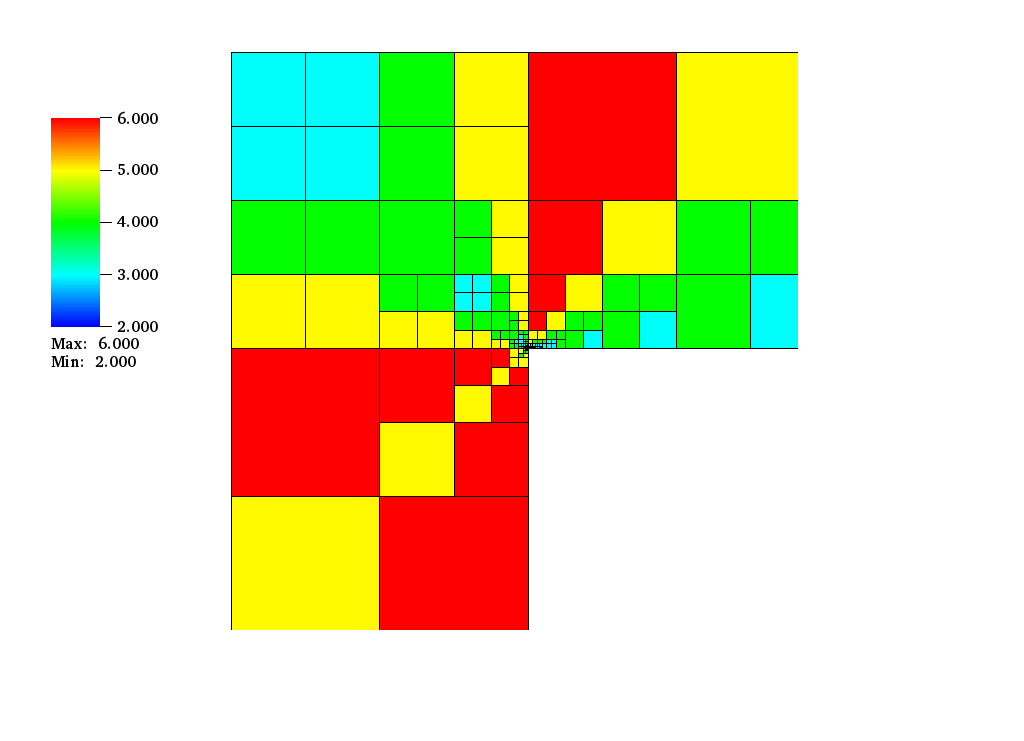} &
                \includegraphics[width=0.45\linewidth]{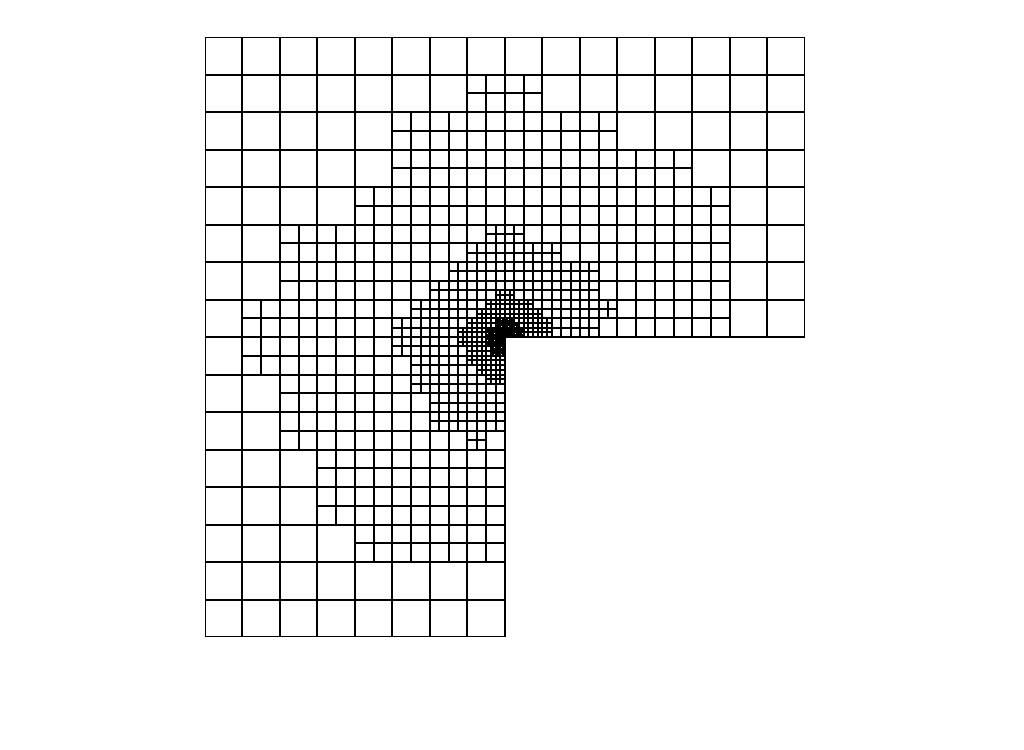} 
        \end{tabular}
        \caption{Example 2. Left: Final mesh after 10 $hp$-adaptive
        refinement steps, with color indicating the polynomial
        degrees. Right: Final mesh after 12 $h$-adaptive refinement
        steps. Both meshes have approximately the same number of
        degrees of freedom.}\label{Ex-2-h_and-hp-ref}
 \end{figure}

 \begin{figure}
        \centering
        \includegraphics[width=0.32\linewidth]{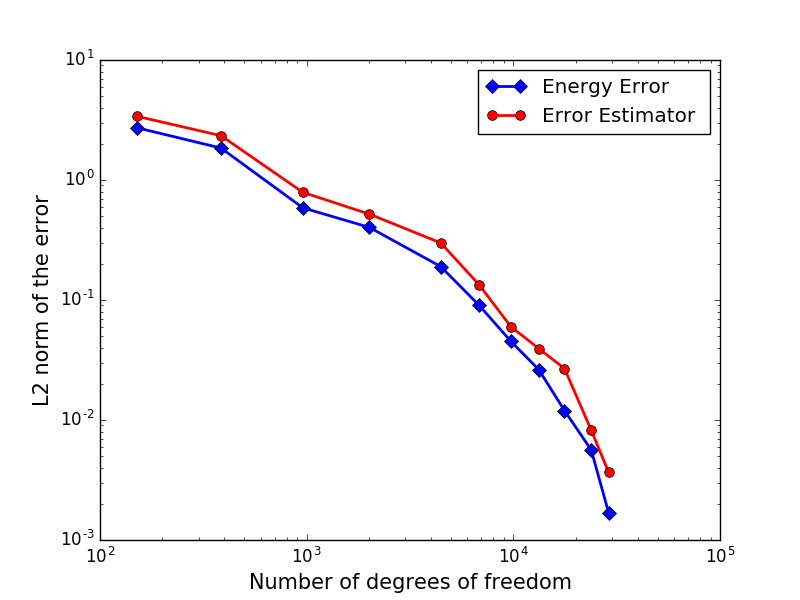} 
        \hfill
        \includegraphics[width=0.32\linewidth]{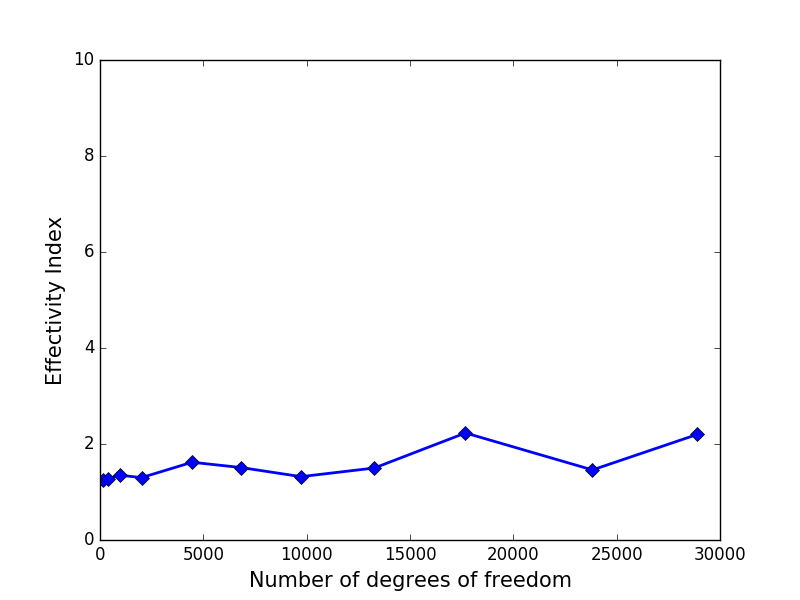}
        \hfill
        \includegraphics[width=0.32\linewidth]{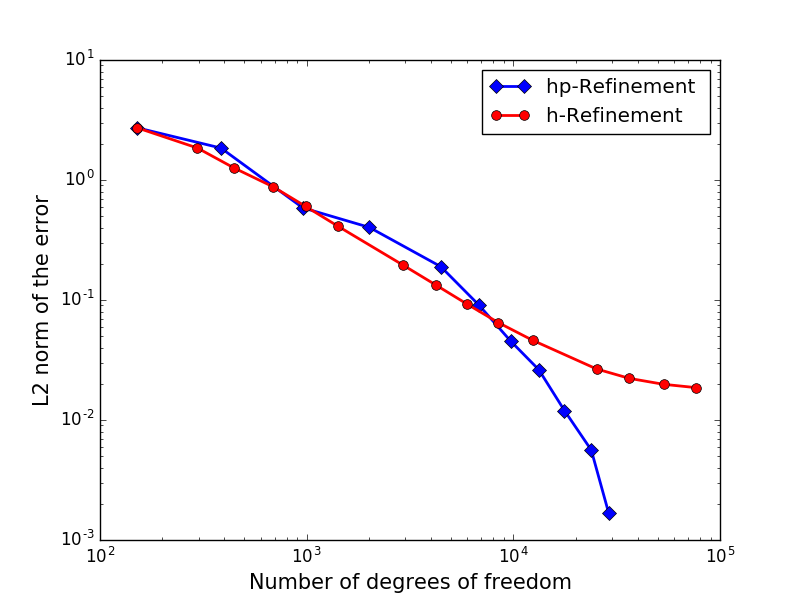}
        \caption{Example 2. Left: Comparison of the energy error and
          the error estimator for an $hp$-adaptive
          computation. Center: Effectivity indices $I_\text{eff}$ for
          the same computation. Right: A comparison of the errors for
          this computation with the errors obtained through pure
          $h$-refinement.}
        \label{error-errorest-eff-index-Ex2} 
 \end{figure}

\subsection{Example 3} \label{sec: Example-4}
As our last example, we consider a less contrived flow field of a
fluid moving through a pipe with a bend. The exact
solution is here not known, but the solution on a
very fine grid is shown in Fig.~\ref{bend-exact-solu-velocity}.

For this case, 
we prescribe homogeneous Dirichlet boundary condition on the sides of
the pipe; for the inlet and outlet, we prescribe parabolic velocity
boundary conditions.  The adaptive algorithms uses $\theta=0.75$ and starts with
28 equally sized cells. The meshes generated by $h$-adaptive refinement are
shown in Figure \ref{bend-exact-solu-velocity}.
\begin{figure}
	\centering
	\begin{tabular}{cc}
		\includegraphics[width=0.45\linewidth]{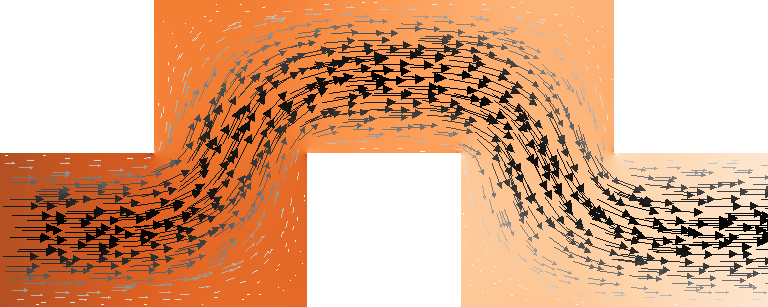} &
		\includegraphics[width=0.45\linewidth]{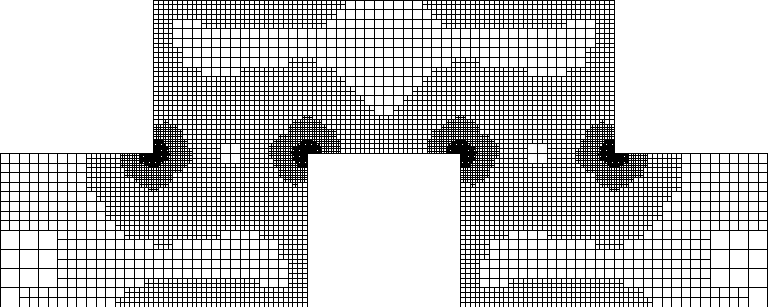}
	\end{tabular}
	\caption{Example 3. Left: Pressure field and velocity vectors
          on a fine mesh. Right: The mesh after 12 $h$-adaptive refinement steps.}
        \label{bend-exact-solu-velocity}
\end{figure}


A comparison between the $h$- and $hp$-adaptively generated meshes in
Fig.s~\ref{bend-exact-solu-velocity} and \ref{bend-hp-mesh} shows the
expected pattern of h-refinement where the solution is not smooth, and
$p$ refinement (if allowed) where the solution is smooth.
Because the exact solution is not known, it is not possible to compare
the exact errors for these two strategies; however, having established
the quality of our error estimator in the previous example, we can
compare how quickly the error estimates are reduced for both
strategies, with results shown in Fig.~\ref{bend-compare-h-hp} --
clearly showing the superiority of $hp$ refinement.

\begin{figure}
	\centering
	\begin{tabular}{c}
		\includegraphics[width=0.75\linewidth]{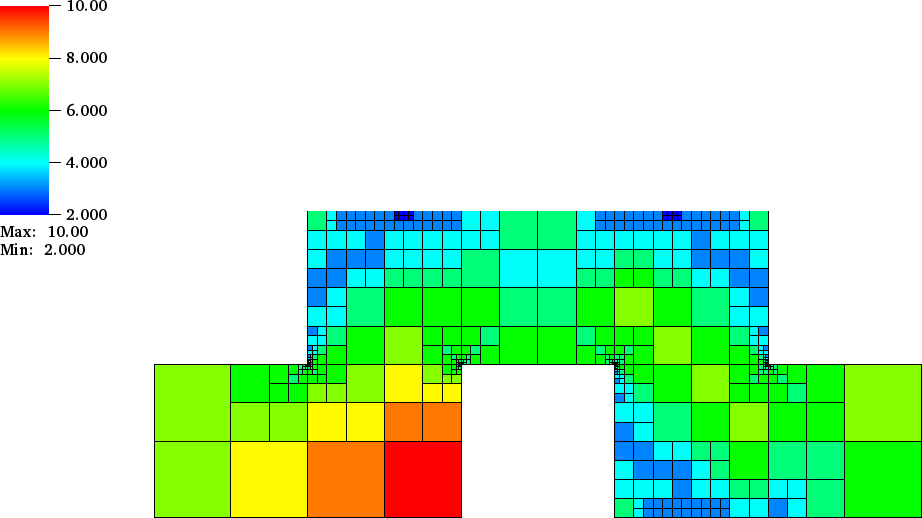} \\
	\end{tabular}
	\caption{Example 3. Mesh generated after 16 $hp$-adaptive steps, where the color bar indicates the polynomial degrees}\label{bend-hp-mesh}
\end{figure}  


\begin{figure}
	\centering
	\begin{tabular}{c}
		\includegraphics[width=0.65\linewidth]{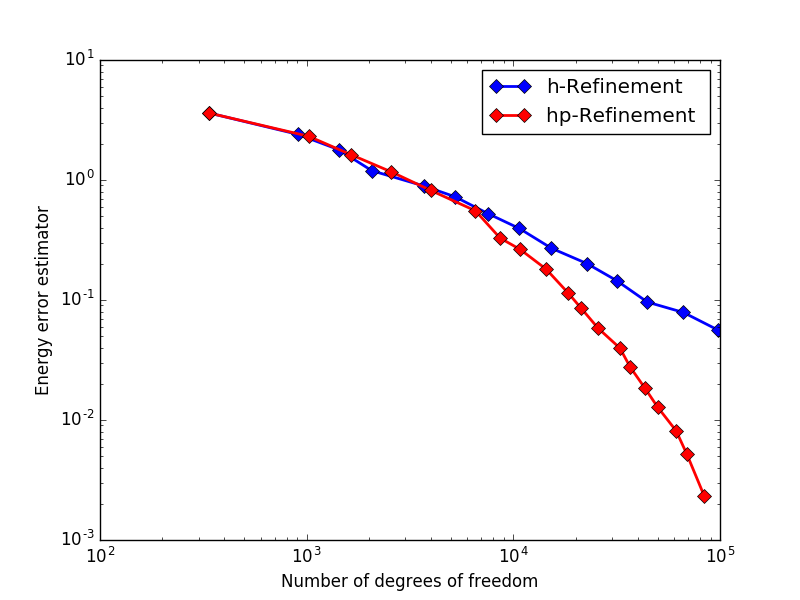} \\
	\end{tabular}
	\caption{Example 3. Comparison of the energy error estimator with $h$- and $hp$-adaptive mesh refinement.}\label{bend-compare-h-hp}
\end{figure}

\section{Conclusion} \label{sec: Conclusion}
In the spirit of previous work by Melenk on other equations (see
\cite{Melenk2005, Melenk2001}), we have here introduced a
residual-based a posteriori error estimator for the Stokes problem for
continuous, $hp$-adaptive finite element methods (AFEM). In particular, we have introduced a family
$\eta_{\alpha}, \alpha \in [0,1]$ of residual based error
estimators. We then proved upper and lower bounds for the
estimators applied to the Stokes problems. We were inspired by
D\"{o}rfler and Heuveline's work \cite{Dorfler2007} for
one-dimensional problems and later work on higher space
dimensions by B\"{u}rg \cite{Buerg_Conv}, and
introduced an $hp$-adaptive refinement algorithm for our
application. In order to decide which refinement gives the best
possible $hp$-refinement, in terms of the largest error reduction, we
solve local patch problems in parallel for each individual cell. The
numerical examples demonstrate the exponential convergence rate for
$hp$-AFEM in comparison with $h$-AFEM. They also show the
efficiency and reliability of the estimator with respect to the norm
of the exact error.

\section{Acknowledgements}
This material is based upon work supported by the U.S. Department of Energy,
Office of Science, under contract number DE-AC05-00OR22725. AG and
WB's work was supported by the National Science
Foundation under award OCI-1148116 as part of the Software Infrastructure for
Sustained Innovation (SI2) program. WB was also supported by the Computational
Infrastructure in Geodynamics initiative (CIG), through the National Science
Foundation under Awards No.~EAR-0949446 and EAR-1550901, administered by The University of California --
Davis.

\bibliographystyle{siam}
\bibliography{hp-Stokes}

\end{document}